\documentclass[12pt]{article}

\usepackage{color}

\newtheorem{definition}{Definition}

\newtheorem{proposition}{Proposition}

\newtheorem{conjecture}{Conjecture}

\usepackage{float}
\usepackage{epsfig}
\usepackage{amssymb}
\usepackage{amsmath}
\usepackage{amsfonts}

\def\qed{\begin{flushright} $\Box$ \end{flushright}}

\def\Dbar{\leavevmode\lower.6ex\hbox to 0pt
{\hskip-.23ex\accent"16\hss}D}

\def\bZ{{\mbox{\bf Z}}}

\def\AF{{\mbox{\rm AF}}}
\def\NAF{{\mbox{\rm NAF}}}
\def\PAF{{\mbox{\rm PAF}}}
\def\GP{{\mbox{\rm GP}}}
\def\NCTP{{\mbox{\rm NCTP}}}
\def\PCTP{{\mbox{\rm PCTP}}}
\def\NGP{{\mbox{\rm NGP}}}
\def\PGP{{\mbox{\rm PGP}}}
\def\Dic{{\mbox{\rm Dic}}}
\def\GF{{\mbox{\rm GF}}}
\def\tr{{\mbox{\rm tr}}}

\def\squareforqed{\hbox{\rlap{$\sqcap$}$\sqcup$}}
\def\qed{\ifmmode\squareforqed\else{\unskip\nobreak\hfil
\penalty50\hskip1em\null\nobreak\hfil\squareforqed
\parfillskip=0pt\finalhyphendemerits=0\endgraf}\fi}
\def\endenv{\ifmmode\;\else{\unskip\nobreak\hfil
\penalty50\hskip1em\null\nobreak\hfil\;
\parfillskip=0pt\finalhyphendemerits=0\endgraf}\fi}
\newenvironment{proof}{\noindent \textbf{{Proof.~} }}{\qed}

\begin{document}

\title
{Negaperiodic Golay pairs and Hadamard matrices}

\author {{Nickolay A. Balonin}
\footnote{ Saint-Petersburg State University of Aerospace Instrumentation, 67, B. Morskaia St., 190000,
Saint-Petersburg, Russian Federation.
email {korbendfs@mail.ru}} ~and
{ Dragomir {\v{Z}. \Dbar}okovi{\'c}}
\footnote{ University of Waterloo, Department of Pure Mathematics and Institute for Quantum Computing, 
Waterloo, Ontario, N2L 3G1, Canada.
e-mail {djokovic@uwaterloo.ca} } }

\date{ \today }

\maketitle

\begin{abstract}
Apart from the ordinary and the periodic Golay pairs, we 
define also the negaperiodic Golay pairs. (They occurred first,  
under a different name, in a paper of Ito.) 
If a Hadamard matrix is also a Toeplitz matrix, we show that it 
must be either cyclic or negacyclic. We investigate the construction of Hadamard (and weighing matrices) from two negacyclic blocks (2N-type). 
The Hadamard matrices of 2N-type are equivalent to negaperiodic 
Golay pairs. We show that the Turyn multiplication of Golay pairs 
extends to a more general multiplication: one can 
multiply Golay pairs of length $g$ and negaperiodic Golay pairs of length $v$ to obtain negaperiodic Golay pairs of length $gv$.
We show that the Ito's conjecture about Hadamard matrices 
is equivalent to the conjecture that negaperiodic Golay pairs 
exist for all even lengths.
\end{abstract}


\section{Introduction}

The Golay pairs (abbreviated as G-pairs, and also known as Golay 
sequences) have been introduced in a note of M. Golay \cite{Golay:1961} published in 1961. Since then they have been studied 
by many reseachers and used in various combinatorial constructions, in particular for the construction of Hadamard 
matrices \cite{SebYam:1992} and \cite[Chapter 23]{deLF:2011}.

The periodic Golay pairs (PG-pairs) made their first appearance, 
under a different name, in a note of the second author 
\cite{Djokovic:DCC:1998} published in 1998. They are equivalent 
to Hadamard matrices built from two circulant blocks (2C-type). 
It is now known that periodic Golay pairs exist for infinitely 
many lengths for which no ordinary Golay pairs are known 
\cite{DK:ADTHM:2015}.

In this paper we complete the picture by defining the 
negaperiodic Golay pairs (NG-pairs). These pairs are 
equivalent to Hadamard matrices built from two negacyclic blocks (2N-type). The NG-pairs were first introduced by N. Ito, under the name of ``associated pairs'', 
in his paper \cite{Ito-IV} published in 2000. 
An intereseting observation is that the ordinary Golay pairs 
are precisely the pairs which are both PG and NG-pairs.

In an earlier paper \cite{Ito-III} Ito proposed a conjecture 
which is stronger than the famous Hadamard conjecture.
It turns out that his conjecture is equivalent to the assertion 
that the NG-pairs exist for all even lengths. This is drastically different from the known facts about ordinary and periodic Golay pairs. Examples of NG-pairs of even length $\le 92$ are listed in \cite{Ito-IV}. As far as we know, no NG-pairs of length 94 have been constructed. 

In section \ref{sec:Toeplitz} 
we show that if a Hadamard matrix is also a Toeplitz matrix, then 
it must be cyclic or negacyclic. As cyclic Hadamard matrices beyond order 4 are not likely to exist, we conjecture that the same holds true for negacyclic Hadamard matrices beyond order 2. 
We have verified the latter conjecture for orders $\le 40$. 
As a substitute for Ito's conjecture we propose the weaker conjecture in which the two negacyclic blocks are replaced by Toeplitz matrices. 

In section \ref{sec:Golay} we define negaperiodic autocorrelation function (NAF) and negaperiodic Golay pairs (NG-pairs). These are 
binary sequences of the same length $v$ whose $\NAF$s add up to 
zero. The length $v$ must be an even integer or 1. 
For the sake of comparisson we recall some facts 
about ordinary and periodic Golay pairs. We show that the Turyn multiplication of G-pairs 
extends to give a multiplication of G-pairs and NG-pairs. 
More precisely, one can multiply G-pairs of length $g$ and 
NG-pairs of length $v$ to obtain NG-pairs of length $gv$.
In particular, one can double the length of any NG-pair.
We also define a natural equivalence relation for NG-pairs.

In section \ref{sec:RDS} we introduce a natural bijection 
$\Phi_v$ from the set of binary sequences of length $v$ onto the set of $v$-subsets of $\bZ_{2v}$. We recall the definition of 
the relative difference families in the cyclic group $\bZ_{2v}$ with respect to the subgroup of order 2. We show that a pair 
of binary sequences of length $v$ is an NG-pair if and only if 
the $\Phi_v$-images of these sequences form a relative 
difference family in $\bZ_{2v}$. We also show that Ito's 
conjecture, which entails the Hadamard matrix conjecture, is 
equivalent to the assertion that NG-pairs exist for all 
even lengths $v$.

There are only a few known infinite series of NG-pairs. 
In sections \ref{sec:Paley-C}, \ref{sec:Paley-1} and 
\ref{sec:Paley-2} we treat two of them, the first and second 
Paley series. First we recall the definition of Paley 
conference matrices (C-matrices). They have order $1+q$ where 
$q$ is an odd prime power. Those for $q\equiv 1 \pmod{4}$ give 
rise to the first Paley series of NG-pairs, with length 
$1+q$. Those for $q\equiv 3 \pmod{4}$ give 
rise to the second Paley series of NG-pairs, with length 
$(1+q)/2$. The main facts that we use are that all Paley 
C-matrices of the same order are equivalent and that each 
of these equivalence classes contains a negacyclic C-matrix.

In section \ref{sec:Ito} we recall that Ito constructed in 
\cite{Ito-III} an infinite series of relative difference sets in 
dicyclic groups (see section \ref{sec:Ito} for the definition).
Hence, this gives an infinite series of NG-pairs to which we 
refer as the Ito series. However, we show that the Ito series is 
contained in the second Paley series.

In section \ref{sec:Williamson} we recall from 
\cite[Corollary 2.3]{Schmidt:1999} the fact that the existence of Ito relative difference sets in the dicyclic group of order $8m$ 
is equivalent to the existence of four generalized Williamson 
matrices of order $m$. We coined the name ``quasi-Williamson 
matrices'' for this type of generalized Williamson matrices. 
The four quasi-Williamson matrices have to be circulants but 
not necessarily symmetric. However, it is required that when 
plugged into the Williamson array they give a Hadamard matrix 
of order $4m$. The known series of four Williamson matrices of 
odd order give rise to the series of NG-pairs. 
As an example, we have computed four quasi-Williamson 
matrices of order 35. It is not known whether quasi-Williamson 
matrices of order 47 exist, and we pose this as an open problem.

In section \ref{sec:WeighingM} we apply NG-pairs to the 
construction of weighing matrices of 2N-type.
For small lengths $v$ we list in the appendices 
\ref{app:A},\ref{app:C} and \ref{app:D} the NG-pairs of the 
first and second Paley series and the Ito series, respectively.

\section
{Block-Toeplitz Hadamard matrices} \label{sec:Toeplitz}

We say that a square matrix $A=[a_{ij}]$, $i,j=0,1,...,v-1$, is 
a Toeplitz matrix if $a_{i,j}=a_{i-1,j-1}$ for $i,j>0$. 
In particular, we will be interested in two classes 
of Toeplitz matrices: cyclic (also known as circulant) 
and negacyclic. The cyclic and negacyclic matrices 
of order $v$ are polynomials in the cyclic and negacyclic 
shift matrix $P$ and $N$, respectively:

\begin{equation} \label{mat-P-N}
P=\left[ \begin{array}{cccccc}
0 & 1 & 0 & \cdots & 0 & 0 \\
0 & 0 & 1 &        & 0 & 0 \\
0 & 0 & 0 &        & 0 & 0 \\
\vdots &  &        &   &   \\
0 & 0 & 0 &        & 0 & 1 \\
1 & 0 & 0 &        & 0 & 0
\end{array} \right],\quad
N=\left[ \begin{array}{cccccc}
0 & 1 & 0 & \cdots & 0 & 0 \\
0 & 0 & 1 &        & 0 & 0 \\
0 & 0 & 0 &        & 0 & 0 \\
\vdots &  &        &   &   \\
0 & 0 & 0 &        & 0 & 1 \\
-1 & 0 & 0 &        & 0 & 0
\end{array} \right].
\end{equation}

\begin{definition} 
A {\em k-Toeplitz matrix} is a square matrix $A$ 
partitioned into square blocks $A_{ij}$, $i,j=1,2,...,k$ 
such that each block $A_{ij}$ is a Toeplitz matrix. 
As a special case $(k=1)$, a square Toeplitz matrix is 
1-Toeplitz. A {\em block-Toeplitz matrix} is a square matrix 
which is k-Toeplitz for some $k$. If each block of 
a k-Toeplitz matrix is cyclic (resp. negacyclic) we 
say that it is {\em k-cyclic} (resp. {\em k-negacyclic}).
We abbreviate ``k-Toeplitz'', ``k-cyclic'', ``k-negacyclic'' with kT, kC, kN, respectively.
\end{definition} 

The k-cyclic Hadamard matrices for $k=1,2,4,8$ have been 
studied extensively 
\cite{BD:2015,Golay:1961,Ryser:1963,Turyn:1974,SebYam:1992}. 
The k-negacyclic ones also have appeared in the literature 
but to much lesser extent \cite{DGS,Ito-IV}. 
In this article we are interested mostly in 
kT-type Hadamard and weighing matrices with $k=1,2,4$. 

For $k=1$ it turns out that Toeplitz Hadamard matrices are 
necessarily cyclic or negacyclic.

\begin{proposition}
If $H=[h_{ij}]$ is a Toeplitz Hadamard matrix of order 
$v \equiv 0 \pmod{4}$, then $H$ is cyclic or negacyclic.
\end{proposition}

\begin{proof} Let $h_i$ be the $(i+1)$th row of $H$, 
$h_i=[h_{i,0},h_{i,1},...,h_{i,v-1}]$. As the rows of $H$
are orthogonal to each other, all dot products of two 
different rows are 0, $h_i \cdot h_j =0$ for $i<j$. 
Let $j\in\{2,3,\ldots,v-1\}$. Then the equality 
$h_0\cdot h_{j-1}=h_1\cdot h_{j}$ simplifies and, by using 
the hypothesis that $H$ is a Toeplitz matrix, we deduce 
that 

\begin{equation} \label{jed-h}
h_{0,v-1} h_{0,v-j} = h_{1,0} h_{j,0},
\quad j=2,3,...,v-1.
\end{equation}

Since the entries of $H$ belong to $\{+1,-1\}$, we have two cases:
$h_{1,0}=h_{0,v-1}$ and $h_{1,0}= -h_{0,v-1}$.

In the former case, from the equations (\ref{jed-h}) we deduce 
that the equality $h_{j,0}=h_{0,v-j}$ holds for all 
$j=1,2,...,v-1$. This means 
that the matrix $H$ is cyclic. Similarly, in the latter case
one can show that $H$ is negacyclic.
\end{proof}

There is a conjecture, attributed to Ryser 
\cite[p. 134]{Ryser:1963}, 
that there exist no cyclic Hadamard matrices of order $>4$. 
We conjecture that the negacyclic analog holds.

\begin{conjecture} \label{neg-conj}
There are no negacyclic Hadamard matrices of order $>2$.
\end{conjecture}

By using a computer we have verified this conjecture for 
orders $\le40$.

For $k=2$ we shall focus on two special classes of 
kT-Hadamard matrices, namely the 2C and 2N-Hadamard matrices having the form

\begin{equation} \label{mat-H}
H = \left[ \begin{array}{cc} 
A & B \\ -B^T & A^T  \end{array} \right].                             
\end{equation}

From now on we refer to 2T, 2C and 2N-matrices having the form 
(\ref{mat-H}) as {\em matrices of 2T-type, 2C-type and 2N-type}, 
respectively. 

We propose the following conjecture.

\begin{conjecture} \label{hip-1}
For each even integer $v>0$ there exists a Hadamard matrix 
of 2T-type and order $2v$.
\end{conjecture}

We shall see in section \ref{sec:RDS} that the stronger 
conjecture below is equivalent to the Ito's conjecture about Hadamard matrices 
(see \cite{ACP:2001,Ito-III,Schmidt:1999,SchmidtTan:2014}).

\begin{conjecture} \label{hip-2}
For each even integer $v>0$ there exists a Hadamard matrix 
of 2N-type and order $2v$.
\end{conjecture}

\section{Three kinds of Golay pairs} \label{sec:Golay}

Let $a=(a_0,a_1,\ldots,a_{v-1})$ be a sequence of integers 
of length $v$. If each $a_i\in\{\pm1\}$ then we say that the 
sequence is {\em binary}. If we allow the sequence to have 
also 0s, then we say that it is {\em ternary}. One defines 
similarly the binary and ternary matrices. 
We shall consider $a$ also as a row-vector. 

There are three kinds of autocorrelation functions that 
we attach to an arbitrary sequence $a$: the ordinary or nonperiodic (AF), the periodic (PAF), and negaperiodic (NAF) autocorrelation functions. They are defined by the formulas

\begin{eqnarray} \label{AF}
\AF_a(k) &=& \sum_{i=0}^{v-k-1} a_i a_{i+k}, \quad k\in\bZ, \\
 \label{PAF}
\PAF_a(k) &=& a\cdot aP^k, \quad k\in\bZ, \\
 \label{NAF}
\NAF_a(k) &=& a\cdot aN^k, \quad k\in\bZ,
\end{eqnarray}
where ``$\cdot$'' is the dot product. 
In (\ref{AF}) we use the convention that $a_i=0$ if 
$i<0$ or $i\ge v$.

Note that for $0\le k<v$ we have
\begin{eqnarray} \label{PAFa}
\PAF_a(k) &=& \AF_a(k)+\AF_a(v-k) \\
 \label{NAFa}
\NAF_a(k) &=& \AF_a(k)-\AF_a(v-k).
\end{eqnarray}

The {\em cyclic shift} and the {\em negacyclic shift} of $a$ 
are given explicitly by
$aP=(a_{v-1},a_0,a_1,\ldots,a_{v-2})$ and 
$aN=(-a_{v-1},a_0,a_1,\ldots,a_{v-2})$, respectively.

Since $N^v=-I$, we have $\NAF_a(k+v)=-\NAF_a(k)$ for all $k$.
It follows immediately from (\ref{NAFa}) that
\begin{equation} \label{NAF-spec}
\NAF_a(v-k)= -\NAF_a(k), \quad 0\le k<v.
\end{equation}
In particular, if $v$ is even then $\NAF_a(v/2)=0$. 
We also mention that $a$, its reverse sequence and the 
negashifted sequence $aN$ all have the same $\NAF$.

If $A$ is the negacyclic matrix with first row $a$, then
$A=\sum_{i=0}^{v-1} a_i N^i$. Further, 
$A^T$ is negacyclic with first row 
$(a_0,-a_{v-1},-a_{v-2},\dots,-a_1)$ and we have
\begin{equation} \label{jed-AAtr}
AA^T = \sum_{k=0}^{v-1} \NAF_a(k) N^k.
\end{equation}
(Similar properties are valid for cyclic matrices.)

Let us define three kinds of complementarity:
\begin{definition} \label{compl}
The integer sequences $a^{(1)},a^{(2)},\ldots,a^{(t)}$, each of 
length $v$, are 

(i) {\em complementary} if $\sum_{i=1}^t \AF_{a^{(i)}}(k)=0$ 
for $k\ne0$;

(ii) {\em P-complementary} if $\sum_{i=1}^t \PAF_{a^{(i)}}(k)=0$ 
for $0<k<v$;

(iii) {\em N-complementary} if $\sum_{i=1}^t \NAF_{a^{(i)}}=0$ for $0<k<v$.
\end{definition}

We now define three kinds of Golay pairs.
\begin{definition} \label{Golay}
A {\em Golay pair (G-pair)}, {\em periodic Golay pair (PG-pair)}, 
{\em negaperiodic Golay pair (NG-pair)} of length $v$ is 
a pair $(a,b)$ of binary sequences of length $v$ which are 
complementary, P-complementary, N-complementary, 
respectively.
We denote by $\GP_v$, $\PGP_v$ and $\NGP_v$ the set of Golay, 
periodic Golay and negaperiodic Golay pairs of length $v$, 
respectively.
\end{definition}

For instance, the pair 
$a=(1,-1,-1,1,-1,-1)$, $b=(1,-1,-1,-1,-1,1)$ is an
NG-pair. It is well known that $\GP_v=\PGP_v=\emptyset$ 
when $v$ is odd and $v>1$. We shall see later that this 
is also true for $\NGP_v$.

The equations (\ref{PAFa}) and (\ref{NAFa}) 
imply that for each $v>0$ we have 
$\GP_v=\PGP_v \cap \NGP_v$.

For the definition of equivalence of G-pairs and of PG-pairs 
see e.g. \cite{Djokovic:DiscrMath:1998} and \cite{BD:2015}, 
respectively. To define the equivalence of NG-pairs 
$(a,b)$ of even length $v$, we introduce the elementary 
transformations which preserve the set of such pairs:\\

(i) reverse $a$ or $b$;

(ii) replace $a$ with $aN$ or $b$ with $bN$;

(iii) switch $a$ and $b$.

(iv) for $k$ relatively prime to $v$, replace 
$a$ and $b$ with the sequences
$(z_ia_{ki \pmod{v}})_{i=0}^{v-1}$ and 
$(z_ib_{ki \pmod{v}})_{i=0}^{v-1}$ respectively, 
where $z_i=1$ if $ki \pmod{2v}<v$ and $z_i=-1$ otherwise.

(v) replace $a_i$ and $b_i$ with $-a_i$ and $-b_i$, 
respectively, for each odd index $i$. \\

We say that two NG-pairs of the same length are 
{\em equivalent} if one can be transformed to the other 
by a finite sequence of elementary transformations.

As an example, we claim that the NG-pairs $(a,b)$ and 
$(c,d)$ of length 10 
\begin{eqnarray*}
&& a=(+,-,-,-,-,+,-,-,-,-),~ b=(+,-,-,+,-,+,-,+,+,-); \\
&& c=(+,-,+,-,+,+,+,-,+,-),~ d=(+,-,-,+,-,+,-,+,+,-);
\end{eqnarray*}
taken from the Appendices C and D, respectively, are equivalent.
(We write $+$ and $-$ for $1$ and $-1$, respectively.)
By applying to $(c,d)$ the elementary transformation (iv) with $k=9$, we obtain the pair $(a,d')$ where 
$d'=(+,-,+,-,-,+,+,-,-,+).$ 
After reversing $d'$ and applying 
the negacyclic shifts, we can transform $d'$ to $b$. 
This proves our claim.

Ito \cite{Ito-IV} gives a list of NG-pairs of length $v=2t$ for 
all odd integers $t\le45$. He also points out that no NG-pair of 
length $94$ is known. Apparently this assertion remains still 
valid.

For lengths $v\le40$, the number of equivalence classes in 
$\GP_v$ and their representatives are known (see e.g. 
\cite{Djokovic:DiscrMath:1998}). Very recently, such 
classification has been carried out in \cite{BD:2015} 
for $\PGP_v$ with $v\le40$.

It is a well-known fact that there is a bijection from  
$\PGP_v$ to the set of 2C-Hadamard matrices of order $2v$.
The image of $(a,b)\in\PGP_v$ is the matrix (\ref{mat-H}) 
in which $a$ and $b$ are the first rows of the circulants
$A$ and $B$. The following is an NG-analog of that result.

\begin{proposition} \label{NG:HM}
If $(a,b)$ is an NG-pair of length $v$ then the matrix 
(\ref{mat-H}), where $A$ and $B$ are the negacyclic blocks with the first rows $a$ and $b$ respectively, is a 2N-type Hadamard matrix of order $2v$. Moreover, this map is a bijection.
\end{proposition}

\begin{proof}
The formula (\ref{jed-AAtr}) implies that if $(a,b)\in\NGP_v$,  then the matrix (\ref{mat-H}) is a 2N-type Hadamard matrix. 
The converse also holds.
\end{proof}

In view of this proposition we can restate Conjecture 
\ref{hip-2} as follows:

\begin{conjecture} \label{hip-3}
$\NGP_v\ne\emptyset$ for all even $v>0$. 
\end{conjecture}

Let us recall (see \cite{DK:ADTHM:2015}) that there are two 
non-equivalent multiplications 
\begin{equation} \label{GPxPGP}
\GP_g \times \PGP_v \rightarrow \PGP_{gv}.
\end{equation}
Interestingly, these two multiplications extend (by using the same formulas) to two multiplications 
\begin{equation} \label{GPxNGP}
\GP_g \times \NGP_v \rightarrow \NGP_{gv}.
\end{equation}

Consequently, in order to prove Conjecture \ref{hip-3}, it 
suffices to consider the case when $v\equiv2\pmod{4}$.

We can generalize the multiplications (\ref{GPxPGP}) and 
(\ref{GPxNGP}) by replacing PG-pairs and NG-pairs with the periodic complementary ternary (PCT) and negaperiodic complementary ternary (NCT) pairs, respectively.
We denote by $\PCTP_{v,w}$ and $\NCTP_{v,w}$ the set of 
PCT-pairs and NCT-pairs of length $v$ and total weight $w$, respectively. (The weight is the number of nonzero terms.)

\begin{proposition} \label{stav:NCTP}
The Turyn multiplication of Golay pairs (see \cite{Turyn:1974}) extends to maps
\begin{eqnarray} \label{GPxPCTP}
&& \GP_g \times \PCTP_{v,w} \rightarrow \PCTP_{gv,gw}, \\
 \label{GPxNCTP}
&& \GP_g \times \NCTP_{v,w} \rightarrow \NCTP_{gv,gw}.
\end{eqnarray}
\end{proposition}

\begin{proof}
The two proofs are essentially the same and we give the proof 
only for the case of NCT-pairs. (This proof is similar to the 
proof of \cite[Proposition 3]{DK:ADTHM:2015}.) 
Given an integer sequence $a=(a_0,a_1,\ldots,a_{v-1})$, we shall  represent it by the polynomial 
$a(z)=a_0+a_1z+\cdots+a_{v-1}z^{v-1}$ in the variable $z$.
The Turyn multiplication $(a,b)\cdot(c,d)=(e,f)$, where 
$(a,b)\in\GP_g$ and $(c,d)\in\GP_v$, is given by the formulas 
\begin{eqnarray} \label{jed:T1}
e(z) &=& \frac{1}{2}(a(z)+b(z))c(z^g)+
\frac{1}{2}(a(z)-b(z))d(z^{-g})z^{gv-g}, \\ \label{jed:T2}
f(z) &=& \frac{1}{2}(b(z)-a(z))c(z^{-g})z^{gv-g}+
\frac{1}{2}(a(z)+b(z))d(z^{g}). 
\end{eqnarray}
The product $(e,f)\in\GP_{gv}$.

Now let us assume that $(c,d)\in\NCTP_{v,w}$. We define the integer sequences $e$ and $f$ of length $gv$ by the same 
formulas (\ref{jed:T1}) and (\ref{jed:T2}), respectively. 
It is easy to see that $e$ and $f$ are ternary sequences. 
Since $(a,b)\in\GP_g$ we have
\begin{equation} \label{jed:Turyn}
a(z)a(z^{-1})+b(z)b(z^{-1})=2g.
\end{equation}
Since $(c,d)\in\NCTP_{v,w}$ we have
\begin{equation} \label{jed:cong}
c(z)c(z)^*+d(z)d(z)^* \equiv w~~ \text{mod } (z^v+1).
\end{equation}
This is an identity in the quotient ring $\bZ[z]/(z^v+1)$, 
which is equipped with the involution ``$*$'' sending 
$z$ to $z^{-1}$. A computation shows that
\begin{eqnarray*} 
4e(z)e(z^{-1})
&=& (a(z)+b(z))(a(z^{-1})+b(z^{-1}))c(z^g)c(z^{-g})+ \\
&& (a(z)-b(z))(a(z^{-1})-b(z^{-1}))d(z^g)d(z^{-g})+ \\
&& (a(z)+b(z))(a(z^{-1})-b(z^{-1}))c(z^g)d(z^g)z^{g-gv}+ \\
&& (a(z)-b(z))(a(z^{-1})+b(z^{-1}))c(z^{-g})d(z^{-g})z^{gv-g}, \\
4f(z)f(z^{-1})
&=& (a(z)-b(z))(a(z^{-1})-b(z^{-1}))c(z^g)c(z^{-g})+ \\
&& (a(z)+b(z))(a(z^{-1})+b(z^{-1}))d(z^g)d(z^{-g})+ \\
&& (b(z)-a(z))(a(z^{-1})+b(z^{-1}))c(z^{-g})d(z^{-g})z^{gv-g}+ \\
&& (a(z)+b(z))(b(z^{-1})-a(z^{-1}))c(z^g)d(z^g)z^{g-gv}.
\end{eqnarray*}
By using (\ref{jed:Turyn}) we obtain that
$$
e(z)e(z^{-1})+f(z)f(z^{-1})=
g(c(z^g)c(z^{-g})+d(z^g)d(z^{-g})).
$$
It follows from (\ref{jed:cong}) that
$$
c(z^g)c(z^{-g})+d(z^g)d(z^{-g})\equiv w~~ \text{mod } (z^{gv}+1)
$$
and so we have
$$
e(z)e(z^{-1})+f(z)f(z^{-1}) \equiv gw~~ \text{mod } (z^{gv}+1).
$$
We conclude that $(e,f)\in\NCTP_{gv,gw}$.
\end{proof}

In the special case when $g=2$ and $(a,b)=((+,-),(+,+))$ 
we obtain a map $\NCTP_{v,w}\to\NCTP_{2v,2w}$ to which we 
refer as ``multiplication by 2''.

\section{Cyclic relative difference families} \label{sec:RDS}

Let us define the map, $\Phi_v$, from the set of binary sequences 
of length $v$ into the set of $v$-subsets of the finite cyclic  
group $\bZ_{2v}$ of integers modulo $2v$. 
If $a=(a_0,a_1,\ldots,a_{v-1})$ is a binary sequence then
\begin{equation} \label{Phi}
\Phi_v(a) = \{i:a_i=1\} \cup \{v+i:a_i=-1\}.                             
\end{equation}
Note that $\Phi_v$ is injective and that its image consists 
of all $v$-subsets $X\subset\bZ_{2v}$ such that
$i-j\ne v$ for all $i,j\in X$.

We also need the definition of relative difference families 
in $\bZ_{2v}$. They are relative to the subgroup $\{0,v\}$ 
of order 2.

\begin{definition} \label{Rel-dif-fam}
The subsets $X_1,X_2,\ldots,X_s$ of $\bZ_{2v}$ form a 
{\em relative difference family} if for each integer 
$m\in\bZ_{2v}\setminus\{0,v\}$ the set of triples 
$\{(i,j,k):\{i,j\}\subseteq X_k,~ i-j\equiv m \pmod{2v} \}$ 
has fixed cardinality $\lambda$, independent of $m$, and there 
is no such triple if $m=v$.
\end{definition}

Note that the parameter $\lambda$ is uniquely determined 
by the obvious equation
\begin{equation} \label{lambda}
\sum_{i=1}^s k_i(k_i-1) = 2\lambda(v-1),                           
\end{equation}
where $k_i=|X_i|$ is the cardinality of $X_i$.

Let us now define the equivalence of relative difference 
families consisting of two $v$-subsets $X,Y\subset\bZ_{2v}$. First we define five types of elementary 
transformations which preserve such families:\\

(i) replace $X$ or $Y$ with its image by the map 
$i\to v-1-i \pmod{2v}$;

(ii) replace $X$ or $Y$ with its image by the map 
$i\to i+1 \pmod{2v}$;

(iii) switch $X$ and $Y$;

(iv) for $k$ relatively prime to $2v$, replace $X$ and $Y$ with 
their images by the map $i\to ki \pmod{2v}$;

(v) replace $X$ and $Y$ with their images by the map which 
fixes the even integers and sends $i\to v+i \pmod{2v}$ 
if $i$ is odd.

\begin{definition} \label{equiv-fam}
Two relative difference families $(X,Y)$ and $(X',Y')$ 
on $\bZ_{2v}$ are {\em equivalent} to each other if one can be 
transformed to the other by a finite sequence of the above 
elementary transformations.
\end{definition}

Let $(a,b)$ be a pair of binary sequences of length $v$ and let 
$X=\Phi_v(a)$ and $Y=\Phi_v(b)$ be the corresponding $v$-subsets 
of $\bZ_{2v}$. We shall see below that $(a,b)$ is an NG-pair 
if and only if $(X,Y)$ is a relative difference family. 
Moreover, the mapping sending $(a,b)\to(\Phi_v(a),\Phi_v(b))$ 
preserves the equivalence classes. 
This follows from the fact that $\Phi_v$ commutes with the 
elementary operations (i-v) defined for NG-pairs in 
section \ref{sec:Golay}and defined above for relative difference families. For instance, if $a'$ is the binary sequence obtained 
from $a$ by applying the elementary transformation (i), then the 
set $\Phi_v(a')$ is obtained from $\Phi_v(a)$ by applying the 
elementary transformation (i) defined above.

As indicated above, the NG-pairs are closely related to relative difference families. The following two propositions make this 
more precise.

\begin{proposition} \label{RDF-seq}
Let  $a^{(1)},a^{(2)},\ldots,a^{(s)}$ be binary sequences of 
length $v$ and let $X_1,X_2,\ldots,X_s$ be the subsets 
of $\bZ_{2v}$ defined by $X_i=\Phi_v(a^{(i)})$. 
If $X_1,X_2,\ldots,X_s$ form a relative difference family in 
$\bZ_{2v}$, then the sequences $a^{(1)},a^{(2)},\ldots,a^{(s)}$ 
are N-complementary. 
\end{proposition}

\begin{proof}
We identify the group ring of $\bZ_{2v}$ 
over the integers with the quotient ring $\bZ[x]/(x^{2v}-1)$ 
of the polynomial ring $\bZ[x]$. The cyclic group $\bZ_{2v}$ is identified with the multiplicative group $\langle x\rangle$ 
by the isomorphism sending $i\to x^i$. The inversion map 
on $\langle x\rangle$ extends to an involutory automorphism 
of $\bZ[x]/(x^{2v}-1)$ which we denote by ``$*$''. The 
subsets $X_i$ are now viewed as subsets of $\langle x\rangle$, 
and will be identified with the sum of their elements in 
$\bZ[x]/(x^{2v}-1)$.

Since the $X_i$ form a relative difference family, we have
\begin{equation} \label{gr-ring-id}
\sum_{i=1}^s X_i X_i^* = \sum_{i=1}^s k_i 
+ \lambda(1+x^v)(x+x^2+\cdots+x^{v-1}).
\end{equation}

The ring of integer negacyclic matrices of order $v$ is 
isomorphic to the quotient ring $\bZ[x]/(x^v+1)$. 
It also has an involutory automorphism ``$*$'' which sends 
$x$ to $x^{-1}$.
Let $f:\bZ[x]/(x^{2v}-1)\to\bZ[x]/(x^v+1)$ be the canonical homomorphism and note that $f(x^v)=-1$.
By applying $f$ to the identity (\ref{gr-ring-id}) we 
obtain that 
$$
\sum_{i=1}^s f(X_i) f(X_i)^* = \sum_{i=1}^s k_i.
$$
Note that $f(X_i)=\sum_{j=0}^{v-1} a^{(i)}_j x^j$ and
$$
f(X_i) f(X_i)^*=\sum_{i=0}^{v-1} \NAF_{a^{(i)}}(j)x^j.
$$
It follows that $\sum_{i=1}^s \NAF_{a^{(i)}}(j)=0$ for 
$j=1,2,\ldots,v-1$, i.e., the sequences 
$a^{(1)},a^{(2)},\ldots,a^{(s)}$ are N-complementary.
\end{proof}

The following partial converse holds.

\begin{proposition} \label{seq-RDF}
Let $a=(a_0,a_1,\ldots,a_{v-1})$ and $b=(b_0,b_1,\ldots,b_{v-1})$ 
be an NG-pair. Then the subsets $X=\Phi_v(a)$ and $Y=\Phi_v(b)$
form a relative difference family in $\bZ_{2v}$ with 
parameter $\lambda=v$.
\end{proposition}

\begin{proof}
We set $R=\bZ[x]/(x^{2v}-1)$, $R^+=\bZ[x]/(x^v-1)$ and 
$R^-=\bZ[x]/(x^v+1)$. Denote the canonical image of $x\in R$ 
in $R^+$ and $R^-$ by $y$ and $z$, respectively. In the proof 
of Proposition \ref{RDF-seq} we have defined the involution 
``$*$'' in $R$ and $R^+$. There is also one in $R^-$ which 
sends $z\to z^{-1}=-z^{v-1}$. These involutions commute with 
the canonical homomorphisms $f:R\to R^-$ and $g:R\to R^+$. 
Note that $R$ is isomorphic to the direct product 
$R^+\times R^-$.

Since $(a,b)$ is an NG-pair, the elements $p,q\in R^-$ 
defined by $p=\sum a_iz^i$ and $q=\sum b_i z^i$ satisfy 
$pp^*+qq^*=2v$. For convenience we identify $X$ with the 
sum of its elements in $R$, and similarly for $Y$. 
Then we have $f(X)=p$ and $f(Y)=q$. It follows that 
$f(XX^*+YY^*-2v)=0$. Thus $XX^*+YY^*-2v$ belongs to the 
kernel of $f$ and, by using the fact that 
$(x^v+1)x^v=x^v+1$ in $R$, we obtain an equality
\begin{equation} \label{X2+Y2}
XX^*+YY^* = 2v + (x^v+1)(c_0+c_1x+\cdots+c_{v-1}x^{v-1}),
\end{equation}
where the $c_i$ are some integers. Since $X=\Phi_v(a)$ 
and $y^v=1$, we have $g(X)=1+y+\cdots+y^{v-1}$. 
Similarly, $g(Y)=g(X)$. Note that $g(X)^*=g(X)$ and 
$g(X)^2=vg(X)$. Hence, by applying $g$ to the equality 
(\ref{X2+Y2}), we obtain that
$$
2v(1+y+\cdots+y^{v-1})=2v+2(c_0+c_1y+\cdots+c_{v-1}y^{v-1}).
$$
We deduce that $c_0=0$ and $c_i=v$ for $i\ne0$. The 
equality (\ref{X2+Y2}) now gives
$$
XX^*+YY^* = 2v + v(x^v+1)(x+x^2+\cdots+x^{v-1}).
$$
Hence $X$ and $Y$ indeed form a relative difference family in 
$\bZ_{2v}$ with the parameter $\lambda=v$.
\end{proof}

It was shown in \cite[Conjecture 1]{ACP:2001} that the Ito's 
conjecture is equivalent to the assertion that for each $t\ge1$ 
there exists a relative difference family $X_1,X_2$ in the cyclic group $\bZ_{4t}$ with $|X_1|=|X_2|=2t$ and $\lambda=2t$. 
By Propositions \ref{RDF-seq} and \ref{seq-RDF} this is in turn 
equivalent to Conjecture \ref{hip-3}.

\section{Paley C-matrices} \label{sec:Paley-C}

A {\em conference matrix} (or {\em C-matrix}) of order $v$ is a 
matrix $C$ of order $v$ whose diagonal entries are 0, the other entries are $\pm1$, and such that $CC^T=(v-1)I$, where $I$ is the identity matrix. There are two well-known necessary conditions for the existence of such matrices. First, $v$ must be even. (We exclude hereafter the trivial case $v=1$.) Second, if 
$v \equiv 2 \pmod{4}$ then $v-1$ must be the sum of two squares. For the existence of negacyclic C-matrices of order 
$v \equiv 4 \pmod{8}$ there is another necessary condition 
\cite{DGS}, namely that 
$v-1 = a^2+2b^2$ for some integers $a$ and $b$.

Two C-matrices are said 
to be {\em equivalent} if they have the same order and one can be obtained from the other by applying a finite sequence of 
the following elementary transformations:
multiplication of a row or a column by $-1$, and interchanging 
simultaneously two rows and the corresponding two columns.

If $v=1+q$ where $q$ is a power of a prime, then Paley 
\cite{Pa} has constructed conference matrices of order $v$. His construction employs essentially the theory of finite fields. 
Let us recall a general definition as given in \cite{DGS}. 
Denote by $V$ a two-dimensional vector space over the Galois 
field $\GF(q)$. Choose any set $X$ of $1+q$ pairwise linearly 
independent vectors of $V$. Denote by $\chi$ the quadratic character of $\GF(q)$. In particular, $\chi(0)=0$. 
(If $q$ is a prime, then $\chi$ is the classical Legendre symbol.) Then the matrix 

\begin{equation} \label{matrica-C_X}
C_X=[\chi(\det(\xi,\eta))], \quad \xi,\eta\in X,
\end{equation}
associated with $X$, is a C-matrix of order $1+q$. 
If $q\equiv 1 \pmod{4}$ then $\chi(-1)=1$ while when 
$q\equiv 3 \pmod{4}$ we have $\chi(-1)=-1$. 
Hence, $C_X$ is symmetric in the former case and 
skew-symmetric in the latter case. 
We refer to $C_X$ as the {\em Paley (conference) matrix}.
It is known that all Paley conference matrices of the same 
order are equivalent to each other \cite{GS:1967}.

In contrast to Conjecture \ref{neg-conj}, there exist an infinite series of negacyclic C-matrices. Indeed, it is shown in 
\cite[Corollary 7.2]{DGS} that each Paley C-matrix is equivalent to a negacyclic C-matrix.

Consequently, the following facts hold.

\begin{proposition} \label{Paley-Nega}
Let $q$ be an odd prime power. Then there exist 

(i)   a negacyclic conference matrix $C$ of order $1+q$; 

(ii)  a 2N-type Hadamard matrix $H$ of order $2(1+q)$;

(iii) an NG-pair of length $1+q$.
\end{proposition}

\begin{proof} In (ii) we can take $H$ to be the matrix 
(\ref{mat-H}) with $A=C+I$ and $B=C-I$. By Proposition 
\ref{NG:HM}, (iii) is equivalent to (ii).
Explicitly, if $(0,c_1,c_2,\ldots,c_q)$ is the 
first row of $C$, then the sequences 
$(1,c_1,c_2,\ldots,c_q)$ and $(-1,c_1,c_2,\ldots,c_q)$ 
form an NG-pair of length $1+q$.
\end{proof}

In Appendix A we list the first rows of the negacyclic Paley 
C-matrices of order $v=1+q\le 128$.

Let $C$ be a negacyclic conference matrix of order $v$ with 
first row $(0,c_1,c_2,\ldots,c_{v-1})$. By a theorem of 
Belevitch (see \cite[Theorem 4.1]{DGS} we have

\begin{equation} \label{jed-Bel}
c_{v/2+j}=(-1)^j c_{v/2-j}, \quad j=1,2,\ldots,v/2-1.
\end{equation}

One may try to find a counter-example to Conjecture 
\ref{neg-conj} as follows. 
Let $q \equiv 3 \pmod{4}$ be a prime power. 
There exists a negacyclic Paley C-matrix $C$ of order $1+q$. 
However, the equations (\ref{jed-Bel}) imply that $C$ is not 
skew-symmetric. Hence $C+I$ is not a Hadamard matrix. 
On the other hand, we know that $C$ is equivalent to a 
skew-symmetric conference matrix $C'$, and so $C'+I$ is a Hadamard matrix. However, $C'+I$ is not negacyclic. It appears 
that $C$ cannot be used to give a negacyclic Hadamard matrix
of order $1+q$.

The two cases $q\equiv 1 \pmod{4}$ and $q\equiv 3 \pmod{4}$ 
in Proposition \ref{Paley-Nega} should be considered separately. 
Indeed, we shall show in section \ref{sec:Paley-2} that in the 
latter case the assertion (iii) of that proposition can be made 
stronger, namely we can replace $1+q$ by $(1+q)/2$.

\section{The first Paley series} \label{sec:Paley-1}

We say that any NG-pair $(a,b)$ of length $v=1+q$ resulting from Proposition \ref{Paley-Nega}, with $q\equiv 1 \pmod{4}$, 
belongs to the {\em first Paley series}. From the proof of that proposition, we recall that $a$ and $b$ are the same sequence except that $b_0=-a_0$.

In this section we assume that $q$ is a prime power and 
that $q\equiv 1 \pmod{4}$. We recall Theorem 7.3 of \cite{DGS}. 

It is easy to verify that if $A$ is a negacyclic matrix of odd order $t$ and $Z$ the diagonal matrix of order $t$ with the diagonal elements $1,-1,1,-1,...$, then the matrix $ZAZ$ is cyclic (and the converse holds). 

\begin{proposition} \label{stav-2C}
Any Paley conference matrix of order $v=1+q \equiv 2 \pmod{4}$, $q$ a prime power, is equivalent to a conference matrix of 
2C-type with symmetric circulant blocks.
\end{proposition} 

Let us give an independent and constructive proof of  Proposition \ref{stav-2C} in the case of negacyclic conference matrices.

\begin{proof}
Let $C$ be a negacyclic conference matrix of order 
$v\equiv 2\pmod{4}$. We shall transform it into the 2N-form, 
and also into the 2C-form with symmetric blocks.

First, we split the first row $c=(0,c_1,c_2,\ldots,c_{v-1})$ of $C$ into two pieces $a=(0,c_2,c_4,\ldots,c_{v-2})$ and 
$b=(c_1,c_3,\ldots,c_{v-1})$. 
One can easily verify that for each integer $k$ we have 
$\NAF_c(2k)=\NAF_a(k)+\NAF_b(k)$. 
It follows that $a$ and $b$ are N-complementary sequences.
Let $A$ and $B$ be the negacyclic matrices with first row 
$a$ and $b$, respectively. 
By plugging the blocks $A$ and $B$ into the array (\ref{mat-H}), we obtain a C-matrix of 2N-type.

Second, we replace $A$ and $B$ with the circulants 
$ZAZ$ and $ZBZ$. The equations (\ref{jed-Bel}) imply that the block $ZAZ$ is symmetric and the first row of $ZBZ$ is symmetric.

Third, we replace the block $ZBZ$ with $ZBZP^m$ where 
$m=(q-1)/4$. Note that $ZBZP^m$ is a symmetric circulant.  There is no need to change the block $ZAZ$. 
By plugging the blocks $ZAZ$ and $ZBZP^m$ into the array 
(\ref{mat-H}), we obtain a C-matrix of 2C-type with symmetric 
blocks.
\end{proof}

Let us give an example. For $q=13$ we have 
$v=14$ and $m=3$. From the table in Appendix A, the first row 
of $C$ is $c=(0,+,+,+,+,+,-,-,+,+,-,+,-,+)$.  
Thus, $a=(0,+,+,-,+,-,-)$ and $b=(+,+,+,-,+,+,+)$. 
The first rows of $ZAZ$ and $ZBZ$ are 
$a'=(0,-,+,+,+,+,-)$ and $b'=(+,-,+,+,+,-,+)$. 
Finally, the first row of the circulant $ZBZP^m$ is 
$b''=(+,+,-,+,+,-,+)$. Thus, the block $ZBZP^m$ is also 
symmetric. By plugging the symmetric circulants $A$ and $B$ 
with first rows $a'$ and $b''$ into the array (\ref{mat-H}), 
we obtain the desired C-matrix of 2C-type.

In Appendix B, for negacyclic Paley C-matrices listed 
in Appendix A and of order $v\equiv 2\pmod{4}$, we list the 
first rows of the symmetric circulant blocks computed by 
the above procedure.

\section{The second Paley series} \label{sec:Paley-2}

In this section we denote by $C$ a negacyclic C-matrix 
of order $n\equiv 0 \pmod{4}$. For convenience we set 
$v=n/2$. We give a very simple construction for NG-pairs of length $v$. In particular we can take $n=1+q$ where 
$q\equiv 3 \pmod{4}$ is a prime power. 
Indeed, as mentioned earlier, we know that any Paley C-matrix 
of order $1+q$ is equivalent to a negacyclic C-matrix. 
We point out that we do not have any other examples of matrices $C$. 

\begin{proposition} \label{prop:Paley-2}
Let $C$ be a negacyclic C-matrix of order $n\equiv 0 \pmod{4}$. 
If $c=(0,c_1,c_2,\ldots,c_{n-1})$ is the first row of $C$, then 
the sequences $a=(1,c_2,c_4,\ldots,c_{n-2})$ and 
$b=(c_1,c_3,\ldots,c_{n-1})$ form an NG-pair of length $v=n/2$. 
\end{proposition}

\begin{proof}
For convenience, we set $a'=(0,c_2,c_4,\ldots,c_{n-2})$. Then $\NAF_{a'}(k)+\NAF_b(k)=\NAF_c(2k)$ for 
$k=1,2,\ldots,v-1$. Since $C$ is a conference matrix, it 
follows from (\ref{jed-AAtr}) that $\NAF_c(k)=0$ for 
$k=1,2,\ldots,n-1$. Hence, $(a',b)$ is an N-complementary pair. However, this is not an NG-pair because the first term of $a'$ 
is 0. 

Let us write $a''=(x,a_1,a_2,\ldots,a_{v-1})$ 
with $a_i=c_{2i}$ for $i=1,2,\ldots,v-1$ and $x$ an 
integer variable. 
We claim that $\NAF_{a''}(k)=\NAF_{a'}(k)$ for $0<k<v$.
Indeed, we have 
$\NAF_{a''}(k)=\AF_{a''}(k)-\AF_{a''}(v-k)=
\NAF_{a'}(k)+x(a_k-a_{v-k})$. By Belevitch's theorem, we
have $a_k=a_{v-k}$ for $0<k<v$ and so $\NAF_{a'}(k)=\NAF_a(k)$. 
Thus our claim is proved.

If we now set $x=1$ then $a''=a$ and we conclude that
$\NAF_a(k)=\NAF_{a'}(k)$ for $0<k<v$. 
Consequently, $(a,b)$ is an NG-pair.
\end{proof}

We say that the NG-pairs constructed in this proposition 
belong to the {\em second Paley series}. 
We say that an NG-pair is a {\em Paley NG-pair} if it belongs 
to the first or the second Paley series.

In Appendix C we list the NG-pairs in the second Paley series 
obtained from the negacyclic C-matrices listed in Appendix A 
with $q\equiv 3 \pmod{4}$.

Out of the 63 odd positive integers $t\le125$, there are exactly 18 for which there is no Paley NG-pair of length $v=2t$.
Let us list these integers:

\begin{equation} \label{losi-2}
23,29,39,43,47,59,65,67,73,81,89,93,101,103,107,109,113,119.
\end{equation}

\section{Ito series} \label{sec:Ito}

There is another series, due to Ito \cite{Ito-III}, of NG-pairs 
of length $(1+q)/2$ when $q\equiv 3 \pmod{4}$ is a prime power. 
However, we will show below that the NG-pairs in this series  
belong to the second Paley series.

For convenience we set  $t=(1+q)/4=v/2$ and let $p$ be the prime 
such that $q=p^n$. The Ito series is
derived from the relative difference sets constructed by 
Ito \cite{Ito-III}. These relative difference sets $R$ have 
parameters $(4t,2,4t,2t)$ and lie in the dicyclic group
\begin{equation} \label{Dic}
\Dic_{8t}=\langle a^{4t}=1,~ b^2=a^{2t},~ bab^{-1}=a^{-1} \rangle
\end{equation}
of order $8t$. The forbidden subgroup is $\langle b^2 \rangle$. 

For convenience we identify a subset $X\subseteq\Dic_{8t}$ with 
the sum of its elements in the group-ring (over $\bZ$) of 
$\Dic_{8t}$. Then we can write $R=R_1+R_2 b$ with $R_1,R_2\subseteq\langle a\rangle$.
The sets $R_1$ and $R_2$ form a relative difference family 
in the cyclic group $\langle a\rangle$ (with the same forbidden 
subgroup). Let us identify $\langle a\rangle$ with $\bZ_{4t}$ 
by the isomorphism sending $a\to1$. It is obvious that $R_1$ and 
$R_2$ are $2t$-subsets of $\bZ_{4t}$. By Proposition 
\ref{RDF-seq}, the binary sequences $X_1=\Phi_v^{-1}(R_1)$ and 
$X_2=\Phi_v^{-1}(R_2)$ form an NG-pair.

We shall now describe a procedure which takes as input the 
integer $t$ and a primitive polynomial $f$ of degree $2n$ over 
the prime field $\GF(p)=\bZ_p$, and gives as output the 
NG-pair arising from the Ito's difference set $R$ in 
$\Dic_{8t}$. This procedure is based on the simplification 
of Ito's construction due to B. Schmidt 
\cite[Theorem 3.3]{Schmidt:1999}.

We construct the Galois field $\GF(q^2)$ by adjoining a root 
$x$ of $f$ to $\bZ_p$. As $q^2-1=((q-1)/2)\cdot(2(q+1))$ and 
$(q-1)/2=2t-1$ and $2(q+1)=8t$ are relatively prime, the 
multiplicative group $\GF(q^2)^*$ is a direct 
product of the subgroups $U$ of order $(q-1)/2$ and $W$ 
of order $2(q+1)$. Note that $U$ is the subgroup of squares 
in $\GF(q)^*$. (Thus we have $Q=U$ for the set $Q$ defined in 
the proof of \cite[Theorem 3.3]{Schmidt:1999}.)

As $f$ is primitive, $x$ generates 
$\GF(q^2)^*$ and the elements $u=x^{8t}$ and $w=x^{2t-1}$ 
generate $U$ and $W$, respectively. Since $x^{(q^2-1)/2}=-1$, 
the element $\alpha=x^{2t}$ satisfies the equation 
$\alpha+\alpha^q=0$, i.e., $\tr(\alpha)=0$ where 
$\tr:\GF(q^2)\to\GF(q)$ is the (relative) trace map. 
We set $v=2t$ and define two 
binary sequences $a=(a_0,a_1,\ldots,a_{v-1})$ and 
$b=(b_0,b_1,\ldots,b_{v-1})$ of length $v$. 
We declare that $a_i=1$ if and only if 
$\tr(\alpha w^{2i})\in U$, and declare that $b_i=1$ if and 
only if $\tr(\alpha w^{2i+1})\in U$.
Then $(a,b)\in\NGP_v$. Note that $a_0=-1$.

We say that the NG-pairs obtained by this procedure belong to  
the {\em Ito series}. They exist for lengths $v=2t$ where 
$q=4t-1$ is a prime power. 

For a sequence $a=(a_0,a_1,\ldots,a_{v-1})$ we say that it is 
{\em quasi-symmetric} if $a_i=a_{v-i}$ for $i=1,2,\ldots,v-1$. 
Note that the negacyclic matrix with first row $a$ is 
skew-symmetric if and only if $a$ is quasi-symmetric and 
$a_0=0$.

The Ito NG-pairs $(a,b)$ have some additional symmetries. Namely, 
$a$ is quasi-symmetric and $b$ is skew-symmetric. 
Both assertions follow from the fact that
$$
\tr(\alpha w^{8t-i})=\tr(\alpha w^{-i})=
\alpha(w^{-i}-w^{-iq})=\alpha w^{-i(q+1)}(w^{iq}-w^i)
=(-1)^i \tr(\alpha w^i).
$$
These symmetry properties were observed by Ito 
\cite[Proposition 6]{Ito-III}, 
as well as the fact that the 2N-type Hadamard matrix constructed 
from the NG-pair $(-a,b)$ is skew-Hadamard. 
(Since the diagonal entries of a skew-Hadamard matrix have to 
be equal to $+1$, we replaced $a$ with $-a$.) 

It follows from these symmetry properties that the negacyclic 
matrix with first row
$$
(0,b_0,a_1,b_1,\ldots,a_{v-1},b_{v-1})
$$
is a conference matrix. This shows that the 
NG-pair $(-a,b)$ belongs to the second Paley series.

In Appendix D we list the NG-pairs of length $v=(1+q)/2\le 154$ in the Ito series, with $q\equiv 3 \pmod{4}$ a prime power.
We have verified directly that each NG-pair listed in 
Appendix C is equivalent to the corresponding NG-pair (the 
one having the same length, $v$) in the list of Appendix D. 

There exist prime powers $q>1$ such that $q\equiv 1 \pmod{4}$ 
and $1+2q$ is also a prime power. For instance, 
$q=5,9,13,29,41$. For such $q$ there exist NG-pairs $(a,b)$ 
and $(c,d)$ of length $1+q$ which belong to the first and
the second Paley series, respectively. 
Then the following question arises: can $(a,b)$ and $(c,d)$ 
be equivalent? (We believe that the answer is negative.)

\section{Quasi-Williamson matrices} 
\label{sec:Williamson}

We say that four binary matrices $A,B,C,D$ of order $t$ are
{\em quasi-Williamson matrices} if they are circulants and 
satisfy the equations
\begin{eqnarray}
& AA^T+BB^T+CC^T+DD^T=4tI, \label{zbir-kvadrata} \\
& AB^T+CD^T=BA^T+DC^T. \label{amicable}
\end{eqnarray} 
This is the cyclic case of a more general definition given in \cite{Schmidt:1999}. In order to avoid a possible confusion, 
we have introduced a different name for this type of matrices.
Note that the above two equations amount to saying that 
the matrix
\begin{equation} \label{WillMat}
\left[ \begin{array}{cccc}
A & B & C & D \\
-B & A & -D & C \\
-C^T & D^T & A^T & -B^T \\
-D^T & -C^T & B^T & A^T \\
\end{array} \right]
\end{equation}
is a Hadamard matrix.

The {\em Williamson matrices} are the special case of 
quasi-Williamson matrices where we require all four blocks 
$A,B,C,D$ to be symmetric, in which case 
the condition (\ref{amicable}) is automatically satisfied. 
Let us mention the following two infinite series of Williamson matrices of order $t$. The first, due to Turyn, exists in 
orders $t=(1+q)/2$, where $q\equiv1 \pmod{4}$ is a prime 
power. Given a conference matrix of 2C-type, see Proposition 
\ref{stav-2C}, with symmetric circulant blocks, say $A$ and $B$, then the matrices $A+I,A-I,B,B$ are four Williamson matrices (this is the Turyn series). 
The second, due to Whiteman, exists in orders 
$t=p(1+p)/2$, where $p \equiv1 \pmod{4}$ is a prime.

In the rest of this section we assume that $t$ is odd. 
Then quasi-Williamson matrices of order $t$ are equivalent 
to relative difference sets in $\Dic_{8t}$ \cite{Schmidt:1999}. 

Let $a,b,c,d$ be the first rows of quasi-Williamson matrices 
of order $t$. We set $z=1$ if $t\equiv1 \pmod{4}$ and $z=-1$ otherwise. We shall describe a procedure which takes as input the
quadruple $a,b,c,d$ and gives as output an NG-pair of 
length $v=2t$. It is based on the proof of 
\cite[Theorem 2.1]{Schmidt:1999}. The subgroup $G\times\langle x\rangle$ of the group $G\times Q_8$, in the mentioned proof, is 
cyclic and is identified with $\bZ_{4t}$.

By using the rows $a$ and $b$, we construct a binary 
sequence $p$ of length $v$ as follows.
Say, $a=(a_0,a_1,\ldots,a_{t-1})$. 
We define two subsets $a',a''$ of $\bZ_t$ by
$a'=\{i:a_i=1\}$ and $a''=\{i:a_i=-1\}$. We define 
similarly the subsets $b',b''\subseteq\bZ_t$. 

For $i=0,1,2,3$ we define the map 
$\Psi_i:\bZ_t\to\bZ_{4t}$ by the formula
\begin{equation} \label{Psi}
\Psi_i(j)=j+t(z(i-j) \pmod{4}).
\end{equation}
It is easy to verify that the set 
$$X=\Psi_0(a')\cup\Psi_1(b')\cup\Psi_2(a'')\cup\Psi_3(b'')$$ 
lies in the image of the map $\Phi_v$ (see (\ref{Phi})).
Finally, we set $p=\Phi_v^{-1}(X)$, which is a binary 
sequence of length $v$.

Similarly, from $c$ and $d$ we construct first a $v$-subset
$Y\subseteq\bZ_{4t}$ and then the binary 
sequence $q=\Phi_v^{-1}(Y)$ of length $v$.
Then $(p,q)\in\NGP_v$.

We remark that the $v$-subsets $X$ and $Y$ form a relative 
difference family in $\bZ_{2v}$ with parameter $\lambda=v$ 
and the forbidden subgroup $\{0,v\}$.

The converse is also true: given an NG-pair $(a,b)$ of length 
$2t$ we can construct quasi-Williamson matrices $A,B,C,D$ of 
order $t$. As an example, we used the NG-pair of length 
$v=70$ given in Appendix D to compute four 
quasi-Williamson matrices $A,B,C,D$ of order 35. The first rows of these matrices (after some cyclic shifts) are:
$$
\begin{array}{lcl}
a&=& [+,-,+,+,-,+,-,+,+,+,+,+,-,+,+,+,-,-,+,+,-,-,-,+,-,-,-,\\&&
-,-,+,-,+,-,-,+],\\
b&=& [+,+,+,+,+,-,-,-,+,-,+,+,-,+,+,+,-,+,+,-,+,+,+,-,+,+,-,\\&&
+,-,-,-,+,+,+,+],\\
c&=& [-,+,-,+,+,-,+,-,+,+,+,-,-,-,-,+,+,+,+,+,+,-,-,+,+,-,+,\\&&
-,+,-,-,-,+,+,-],\\
d&=& [-,+,+,-,-,-,+,-,+,-,+,+,-,-,+,+,+,+,+,+,-,-,-,-,+,+,+,\\&&
-,+,-,+,+,-,+,-],
\end{array}
$$
respectively. The blocks $A,B,C,D$ satisfy the equations 
(\ref{zbir-kvadrata}) and (\ref{amicable}), and when plugged 
into the array (\ref{WillMat}) we do get a Hadamard matrix. 
Moreover, $A$ is of skew-type, while $B$ is symmetric,  
and $d$ is the reverse of $c$. 
Note also that $AB^T-BA^T=(A-A^T)B\ne0$, and so $A,B,C,D$ are not 
matrices of Williamson type according to 
\cite[Definition 3.3]{SebYam:1992}.

It is known that Williamson matrices of odd order $t$ exist 
for $t=23,29,39,43$, see e.g. \cite{HKT}. After removing these 
integers, the list (\ref{losi-2}) reduces to
\begin{equation} \label{losi-3}
47,59,65,67,73,81,89,93,101,103,107,109,113,119.
\end{equation}

Let us single out the smallest case.

{\bf Open Problem~} Do quasi-Williamson matrices of order 47 exist? Equivalently, do NG-pairs of length 94 exist?

The above mentioned facts have been known since 1999 (see \cite{Schmidt:1999,Ito-IV}) and apparently no progress
has been made so far in the search for NG-pairs of  
order $v=2t$, for $t$ in the above list. 
For generalizations where the cyclic group $\bZ_{4t}$ is replaced by more general finite abelian groups see \cite{SchmidtTan:2014}.

Since the known infinite series of NG-pairs are rather sparse, 
it is hard to believe that NG-pairs exist for all even lengths. 
In other words, in our opinion Ito's conjecture is likely to be 
false.

\section{Weighing matrices of 2N-type} 
\label{sec:WeighingM}

A {\em weighing matrix} of order $n$ and {\em weight} $w$ 
(abbreviated as $W(n,w)$) is a matrix $W$ of order $n$ with entries in $\{0,\pm1\}$ such that $WW^T=wI$. 
In this section we discuss the existence of weighing matrices 
of 2N-type.

Note that C-matrices of order $v$ are $W(v,v-1)$. It is 
known that there are no cyclic $W(v,v-1)$ for $v>2$
\cite{StaMul}. 
On the other hand there are infinitely many negacyclic 
$W(v,v-1)$. Indeed each Paley C-matrix is equivalent to 
a negacyclic C-matrix. 
It has been conjectured \cite{DGS} that there are no 
negacyclic C-matrices of even order $v\ne1+q$, $q$ a prime 
power. This conjecture has been verified for $v\le226$.
However, there exist C-matrices of 2N-type whose order $v$ 
is not of that form. For instance, they exist for 
$$
v=16,40,52,56,64,88,96,120,136,144,160.
$$
(See part (iii) of the proposition below.)

We have four infinite series of 2N-type weighing matrices. 

\begin{proposition} \label{teor:tezmat}
Let $q$ be an odd prime power. Then there exist weighing matrices of 2N-type: 

(i) $W(1+q,q)$;

(ii) $W(2+2q,2q)$;

(iii) if $q\equiv 3 \pmod{4}$, $W(2+2q,1+2q)$ and $W(4+4q,2+4q)$. 
\end{proposition}
\begin{proof}
(i) If $q\equiv 1 \pmod{4}$, this was shown in the proof of 
Proposition \ref{stav-2C}. Otherwise the claim follows from 
the fact, proven in section \ref{sec:Ito}, that there exists an 
NG-pair $(a,b)$ of length $(1+q)/2$ with $a$ quasi-symmetric. 
Let $A$ and $B$ be the negacyclic matrices with first rows $a$ and $b$. We may assume that $a_0=1$, then the matrix 
(\ref{mat-H}) is skew-Hadamard of 2N-type. By replacing the 
diagonal entries with 0s, we obtain a $W(1+q,q)$. 

(ii) This follows from (i) because we can ``multiply by 2''.

(iii) Let $(a,b)$ be an Ito NG-pair of length $(1+q)/2$. 
By multiplying by 2, we obtain an NG-pair $(a',b')$ of length 
$1+q$ with $a'=(1,a'')$ quasi-symmetric.
Consequently, the pair $((0,a''),b')$ is N-complementary.
The corresponding 2N-type matrix (\ref{mat-H}) is a C-matrix 
of order $2+2q$. Multiplying by 2 we obtain also an 
$W(4+4q,2+4q)$.
\end{proof}

This proposition covers all weighing matrices $W(4n,4n-1)$ and
$W(4n,4n-2)$ of 2N-type, for $n\le 50$ except for 
$$
n=9,13,19,23,25,28,29,31,37,39,43,44,46,47,48,49
$$
and
$$
n=11,17,18,26,29,33,35,38,39,43,46,47,50,
$$
respectively. We have constructed five of these matrices:

$$
\begin{array}{rl}
n & a~\&~b \\
\hline \\
11 &
[0,-,-,+,-,-,-,-,-,+,+,+,-,+,+,-,+,-,+,+,+,-], \\&
[0,+,-,-,-,-,+,-,-,+,-,+,+,+,+,-,-,-,+,-,+,+] \\
13 & 
[0,+,+,-,+,-,+,+,+,+,-,+,+,-,+,+,-,+,+,+,+,-,+,-,+,+],\\&
[+,+,+,-,+,-,-,-,+,-,-,-,+,+,+,-,+,+,-,-,+,+,+,+,-,-] \\
17 &
[0,-,-,-,+,-,-,+,-,+,-,-,-,+,-,+,+,+,-,+,+,+,+,-,+,+,+,+,+, \\&
-,-,-,+,-],~[0,-,+,+,+,+,-,+,-,-,-,+,+,+,+,-,-,+,+,-,-,+,-, \\&
+,+,-,+,+,+,+,-,+,-,-] \\
18 & 
[0,+,+,-,-,-,+,-,-,+,-,-,-,-,-,-,+,+,+,-,+,+,-,+,-,+,-,-,-, \\&
+,+,+,-,+,+,-],~[0,+,+,-,-,+,+,+,+,-,+,-,-,-,+,-,+,+,+,-,+, \\&
+,+,+,-,+,+,+,+,-,+,-,-,+,+,-] \\
26 &
[0,+,+,+,+,-,-,-,+,-,-,-,+,+,+,+,+,-,+,-,-,+,+,-,+,-,-,-,+, \\&
+,+,+,-,+,+,-,+,+,+,+,+,+,-,-,+,+,-,-,+,-,+,-],~[0,+,+,+,+, \\&
+,-,-,+,+,-,-,+,-,+,-,+,+,+,-,-,-,+,-,+,+,-,+,+,+,+,-,-,+, \\&
+,+,+,-,+,-,+,+,-,+,+,+,-,+,+,-,+,-].
\end{array} 
$$

Multiplication by Golay pairs may be used to construct other
series of weighing matrices of 2N-type. 

In Appendix E we list weighing matrices $W(4n,4n-2)$ of 4C-type 
for odd $n\le 21$. They can be easily converted to 4N-type by replacing each circulant block $X$ of order $n$ with the negacyclic block $ZXZ$.

\section{Acknowlegdements}
The second author wishes to acknowledge generous support by NSERC. This work was made possible by the facilities of the Shared Hierarchical Academic Research Computing Network (SHARCNET) and Compute/Calcul Canada.

\section{Appendix A} \label{app:A}

For even integers $v=1+q\le128$, with $q=p^n$ a power of a prime 
$p$, we give the first row $c$ of a negacyclic conference matrix 
$C$ of order $v$ belonging to the equivalence class of Paley 
conference matrices. The algorithm is described in section 
\ref{sec:Paley-C}, it is based on \cite[Corollary 7.2]{DGS}. We also record the primitive polynomial $f(x)$ of degree $2n$ over $\GF(p)$ used in the computation.

$$
\begin{array}{rl}
 v & f(x);~p,q \text{ and the first row } c \\
\hline \\
 4 & x^2+x+2;~ p=q=3 \\ &
[0,+,-,-] \\
 6 & x^2+x+2;~ p=q=5 \\ &
[0,+,+,+,-,+] \\
 8 & x^2+x+3;~ p=q=7 \\ &
[0,+,-,-,-,+,-,-] \\
10 & x^4+x^3+2;~ p=3,~ q=9 \\ &
[0,+,-,-,-,-,+,-,+,+] \\
12 & x^2+x+7;~ p=q=11 \\ &
[0,+,-,+,-,-,+,+,-,-,-,-] \\
14 & x^2+x+2;~ p=q=13 \\ &
[0,+,+,+,+,+,-,-,+,+,-,+,-,+] \\
18 & x^2+x+3;~ p=q=17 \\ &
[0,+,+,+,-,+,+,+,+,-,-,+,-,+,+,+,-,+] \\
20 & x^2+x+2;~ p=q=19 \\ &
[0,+,-,-,-,-,-,+,-,-,+,+,-,-,-,+,-,+,-,-] \\
24 & x^2+x+7;~ p=q=23 \\ &
[0,+,-,-,+,+,+,+,+,+,-,+,+,-,-,-,+,-,+,-,+,+,-,-] \\
26 & x^4+x^3+x+3;~ p=5,~ q=25 \\ &
[0,+,-,-,-,+,+,+,-,-,-,-,-,-,+,-,+,-,+,+,-,+,+,-,+,+] \\
28 & x^6+x^5+2;~ p=3,~ q=27 \\ &
[0,+,+,-,-,+,-,-,+,+,+,+,+,+,+,-,+, -,+,-,+,+,-,-,-,+,+,-] \\
30 & x^2+x+3;~ p=q=29 \\ &
[0,+,+,-,+,+,+,+,+,+,-,-,-,-,+,+,-,-,+,-,+,+,-,+,-,+,-,-, \\ &
-,+] \\
32 & x^2+x+12;~ p=q=31 \\ & 
[0,+,-,+,-,+,+,+,+,-,-,-,+,+,-,-,-,+,-,-,+,+,-,+,+,-,+,-, \\ & 
 -,-,-,-] \\
38 & x^2+x+5;~ p=q=37 \\ & 
[0,+,+,+,+,+,-,+,-,+,+,-,+,+,+,+,-,-,+,+,-,-,+,+,-,+,-,-, \\ &
-,+,+,+,+,+,-,+,-,+] \\
42 & x^2+x+12;~ p=q=41 \\ & 
[0,+,+,-,+,-,+,+,+,+,-,+,-,-,+,+,+,-,-,-,+,-,-,-,+,-,-,+, \\ &
-,-,+,+,+,+,-,+,-,-,-,-,-,+] \\
44 & x^2+x+3;~ p=q=43 \\ & 
[0,+,-,+,-,+,-,+,+,-,-,+,-,-,+,-,+,+,-,+,+,+,-,-,+,-,-,-, \\ &
+,+,+,+,-,-,-,+,+,-,-,-,-,-,-,-] \\
48 & x^2+x+13;~ p=q=47 \\ & 
[0,+,-,-,+,+,+,-,-,+,+,+,-,+,-,+,-,+,+,+,+,-,+,-,-,+,+,+, \\ &
+,-,+,-,-,-,-,-,-,-,+,-,-,+,+,-,+,+,-,-] \\
\end{array} 
$$

$$
\begin{array}{rl}
50 & x^4+x^3+x^2+3;~ p=7,~ q=49 \\ & 
[0,+,-,-,-,+,+,-,+,-,+,-,-,-,+,-,+,-,-,+,+,+,-,-,-,-,+,-, \\ &
+,+,-,+,+,-,-,-,-,-,+,-,-,-,-,-,-,+,+,-,+,+] \\
54 & x^2+x+5;~ p=q=53 \\ & 
[0,+,+,+,+,+,-,+,-,-,-,+,-,+,+,-,-,+,-,+,-,+,+,-,+,+,-,-, \\ &
+,+,-,-,-,+,+,+,+,+,+,-,-,+,+,+,+,-,+,+,+,+,-,+,-,+] \\
60 & x^2+x+2;~ p=q=59 \\ & 
[0,+,-,-,+,-,-,+,+,-,-,-,-,-,-,+,+,+,-,+,-,+,-,-,-,+,-,+, \\ &
+,-,+,+,+,-,-,-,-,+,-,-,-,-,-,-,+,-,-,+,-,+,-,+,+,-,-,+,  \\ &
+,+,-,-] \\
62 & x^2+x+2;~ p=q=61 \\ & 
[0,+,+,+,+,+,+,-,+,-,-,-,+,+,-,-,-,+,-,-,-,-,+,+,+,-,-,-, \\ &
+,-,+,+,-,-,-,-,+,-,-,+,-,-,+,-,+,+,+,-,+,+,-,-,+,-,-,-,  \\ &
-,+,-,+,-,+] \\
68 & x^2+x+12;~ p=q=67 \\ & 
[0,+,-,+,+,-,+,+,+,+,-,+,+,+,-,+,+,-,-,-,-,-,+,+,-,-,-,-, \\ &
-,+,-,+,+,-,+,+,+,-,-,-,-,+,-,+,-,-,+,+,-,+,-,+,+,-,-,-,  \\ &
+,-,-,-,+,-,+,+,+,-,-,-] \\
72 & x^2+x+11;~ p=q=71 \\ & 
[0,+,-,-,-,+,-,-,-,-,-,+,+,+,+,-,-,-,+,+,-,-,+,-,+,-,+,+, \\ &
+,+,-,+,+,-,+,+,+,-,+,+,+,-,-,-,+,-,+,+,+,+,+,+,-,-,+,+,  \\ &
-,+,+,-,+,-,-,+,-,+,-,-,-,+,-,-] \\
74 & x^2+x+11;~ p=q=73 \\ & 
[0,+,+,-,-,+,-,+,+,+,-,+,-,+,-,+,-,-,+,-,-,-,-,+,-,+,+,+, \\ &
-,-,+,+,+,-,-,-,-,+,+,-,+,-,-,+,-,-,+,+,-,+,+,+,+,-,+,-,  \\ &
-,-,+,+,+,+,+,+,+,+,-,+,+,+,+,-,-,+] \\
80 & x^2+x+3;~ p=q=79 \\ & 
[0,+,-,-,+,+,-,+,-,+,-,+,-,-,+,-,-,-,+,-,-,-,-,+,-,+,+,-, \\ &
-,-,-,+,-,+,-,-,+,+,+,-,-,+,+,-,+,+,-,-,-,-,-,+,-,+,+,-,  \\ &
-,-,-,+,-,+,+,+,-,+,+,+,-,-,-,-,-,-,-,-,+,+,-,-] \\
82 & x^8+x^5+2;~ p=3,~ q=81 \\ & 
[0,+,+,-,-,+,-,+,+,+,-,+,+,-,+,-,+,-,+,+,+,-,+,+,-,-,-,-, \\ &
+,+,+,+,+,-,+,+,-,+,+,+,+,-,-,+,-,+,+,+,-,-,-,+,-,+,-,-,  \\ &
+,-,+,+,-,-,-,+,-,-,-,-,-,-,-,+,+,+,-,+,+,+,+,-,-,+] \\
84 & x^2+x+2;~ p=q=83 \\ & 
[0,+,-,-,+,-,+,+,+,-,+,+,+,-,-,+,-,-,+,+,-,+,+,-,+,-,+,-, \\ &
+,-,-,+,+,-,+,-,-,-,-,-,+,-,+,+,+,+,-,+,-,+,+,+,+,-,-,+,  \\ &
+,+,+,+,+,+,+,-,-,-,+,+,-,-,-,+,+,-,+,+,+,-,+,+,+,+,-,-] \\
\end{array} 
$$

$$
\begin{array}{rl}
90 & x^2+x+6;~ p=q=89 \\ & 
[0,+,+,+,+,-,+,-,+,+,+,-,-,-,+,-,-,+,-,+,+,+,-,+,+,+,-,-, \\ &
+,+,-,+,-,-,-,+,+,-,+,-,+,+,-,+,-,+,+,+,+,+,-,-,-,-,-,+,  \\ &
+,-,+,+,+,+,-,-,+,+,-,+,+,+,-,+,+,+,+,-,-,-,+,-,-,+,-,-,  \\ &
-,-,-,+,-,+] \\
98 & x^2+x+5;~ p=q=97 \\ & 
[0,+,+,+,+,+,-,+,+,-,+,+,-,+,-,-,+,+,-,-,-,+,+,-,+,-,+,-, \\ &
-,+,-,+,+,+,-,+,-,-,-,-,-,+,+,+,-,+,+,+,-,+,+,+,-,+,+,+,  \\ &
-,+,+,-,+,-,+,+,+,+,-,+,+,+,+,-,-,-,-,-,-,+,+,-,+,+,-,-,  \\ &
+,+,+,+,-,-,-,+,+,+,-,+,-,+] \\
102 & x^2+x+3;~ p=q=101 \\ & 
[0,+,+,-,+,+,+,+,-,-,+,-,-,-,-,-,+,+,+,+,+,+,+,+,-,+,+,-, \\ &
+,+,-,+,-,-,-,+,+,+,-,-,-,-,-,+,-,+,+,+,-,-,+,+,-,-,+,+,  \\ &
-,+,+,+,+,-,+,-,+,+,-,+,+,-,+,+,+,+,-,-,-,+,+,+,-,+,-,+,  \\ &
-,+,-,-,+,-,+,-,-,-,+,+,-,+,-,-,-,+] \\
104 & x^2+x+5;~ p=q=103 \\ & 
[0,+,-,-,+,-,-,+,-,+,-,+,-,-,-,-,+,-,-,-,+,-,+,+,+,-,+,-, \\ &
-,-,-,-,-,+,+,-,+,+,-,-,+,+,+,-,+,-,-,+,+,+,+,-,-,+,+,-,  \\ &
+,-,-,+,+,+,+,-,+,+,-,-,+,+,+,-,-,+,-,+,-,+,+,+,+,-,+,+,  \\ &
+,+,-,+,+,+,-,+,-,-,-,-,-,-,-,+,+,+,-,-] \\
108 & x^2+x+5;~ p=q=107 \\ & 
[0,+,-,-,+,+,-,+,-,+,-,-,-,-,-,+,+,+,+,+,-,-,-,-,+,-,+,-, \\ &
-,-,+,-,-,-,-,+,-,+,-,-,+,+,-,-,+,-,-,+,+,+,-,-,-,+,-,-,  \\ &
-,+,-,-,+,-,-,+,+,+,-,-,+,+,-,-,-,-,-,+,-,+,+,+,-,+,+,+,  \\ &
+,+,-,+,-,-,+,-,+,-,-,+,-,+,-,-,-,-,-,-,+,+,-,-] \\
110 & x^2+x+6;~ p=q=109 \\ & 
[0,+,+,+,-,-,-,+,-,+,+,+,-,+,-,+,-,+,+,+,-,+,+,-,+,-,-,+, \\ &
+,-,+,+,-,-,-,+,-,+,-,-,+,+,-,+,+,+,+,-,-,-,-,-,-,+,+,+,  \\ &
-,+,+,-,+,-,+,-,-,+,-,+,+,+,-,-,+,+,+,+,+,-,+,+,-,-,-,+,  \\ &
+,-,-,-,-,+,+,+,-,+,+,+,+,+,+,+,-,+,+,+,+,-,+,+,-,+] \\
114 & x^2+x+10;~ p=q=113 \\ & 
[0,+,+,+,-,-,-,-,+,-,-,+,+,-,+,+,-,+,-,-,-,+,-,+,-,-,-,-, \\ &
-,+,+,+,+,+,-,+,+,-,+,-,+,-,-,-,-,-,+,+,-,-,+,-,-,-,+,-,  \\ &
-,+,+,-,-,-,+,-,-,-,+,+,-,-,+,-,+,-,-,-,-,-,-,+,+,+,-,+,  \\ &
-,+,+,-,+,-,+,+,+,+,+,-,+,+,+,+,-,-,-,+,+,-,-,-,+,-,+,+,  \\ &
-,+] \\
\end{array} 
$$

$$
\begin{array}{rl}
122 & x^4+x^3+8;~ p=11,~ q=121 \\ & 
[0,+,-,-,-,-,-,-,-,+,+,-,+,-,-,-,+,-,-,+,+,-,+,-,-,+,+,-, \\ &
+,-,-,-,-,+,-,+,-,-,-,-,+,-,-,-,+,+,+,+,-,+,-,-,-,+,+,+,  \\ &
-,-,+,-,-,-,+,-,-,-,+,+,-,+,+,-,+,+,+,+,-,+,-,-,+,-,-,-,  \\ &
+,-,+,+,+,+,+,-,+,-,-,-,-,+,+,-,-,-,-,+,+,-,-,-,+,-,-,-,  \\ &
-,+,+,-,+,-,+,-,+,+] \\
128 & x^2+x+3;~ p=q=127 \\ & 
[0,+,-,-,-,+,-,+,+,-,-,-,+,+,-,-,-,-,+,-,-,-,-,+,-,+,+,+, \\ &
+,+,-,+,+,-,+,+,-,-,-,+,+,+,+,-,+,+,+,-,+,+,+,+,+,+,-,-,  \\ &
+,-,+,-,+,+,-,-,-,+,-,-,+,+,+,+,+,+,-,-,+,-,+,-,+,+,+,-,  \\ &
+,+,+,-,+,-,-,+,-,-,+,+,+,-,-,-,+,-,+,-,-,-,-,+,-,+,+,+,  \\ &
-,+,-,-,+,+,-,+,+,-,-,-,-,+,-,-] \\
\end{array} 
$$

\section{Appendix B} \label{app:B}

For even integers $v=1+q\le128$, with $q \equiv 1 \pmod{4}$ 
a prime power, we give a 2C-type conference matrix of order $v$
with symmetric blocks $A$ and $B$ which belongs to the equivalence class of Paley conference matrices. The algorithm is described in section \ref{sec:Paley-1}. Since the blocks $A$ and $B$ are symmetric circulants of odd order $v/2$, we record only the first $(v+2)/4$ elements of their first rows $a$ and $b$.

We recall that $A+I,A-I,B,B$ are four Williamson matrices of 
order $v/2$ belonging to the Turyn series.

$$
\begin{array}{rl}
v   & \text{first rows of $a$ and $b$ (truncated)} \\
\hline \\
6   & [0,-],~[-,+] \\
10  & [0,+,-],~[-,+,+] \\
14  & [0,-,+,+],~[+,+,-,+] \\
18  & [0,-,-,-,+],~[-,-,+,-,+] \\
26  & [0,+,-,-,-,+,-],~[-,+,-,-,+,+,+] \\
30  & [0,-,+,-,+,+,-,-],~[-,-,+,+,-,+,+,+] \\
38  & [0,-,+,+,-,-,+,-,-,-],~[-,-,-,+,+,+,-,+,-,+] \\
42  & [0,-,+,-,+,+,-,-,+,+,+],~[-,+,-,-,-,-,+,-,-,+,+]  \\
50  & [0,+,-,-,+,-,-,-,+,+,+,+,-],~
      [-,+,+,-,-,+,-,+,-,+,+,+,+] \\
54  & [0,-,+,+,-,+,-,-,-,+,-,-,+,+],~
      [+,+,+,+,-,+,+,+,-,-,-,+,-,+] \\
62  & [0,-,+,-,+,+,+,+,-,+,-,-,+,+,+,-], \\&
      [-,-,+,-,-,-,+,+,+,+,+,-,+,+,-,+] \\ 
74  & [0,-,-,+,+,+,-,+,-,-,-,+,-,-,-,-,+,+,-], \\&
      [+,+,-,-,-,-,+,-,-,+,-,-,+,-,+,-,+,+,+] \\
82  & [0,-,-,+,+,+,+,-,+,-,+,-,-,+,+,-,+,-,-,-,+], \\&
      [-,-,+,-,-,-,+,+,-,-,-,-,-,+,-,-,+,-,+,+,+] \\
90  & [0,-,+,-,+,-,-,-,-,+,+,+,+,+,+,+,-,+,+,-,+,+,-], \\&
      [+,-,+,+,-,-,-,-,+,+,+,-,+,-,+,+,-,+,+,+,-,-,+] \\
98  & [0,-,+,+,+,-,-,+,+,+,-,-,+,-,-,+,+,+,-,+,-,-,-,-,-], \\&
      [+,-,+,-,+,+,-,-,+,-,+,+,-,+,+,+,+,+,+,-,-,-,+,-,+] \\
102 & [0,-,+,-,-,-,-,+,+,-,+,-,-,-,+,+,-,+,+,+,-,+,-,-,-,-],\\&
      [-,-,-,+,-,-,+,+,-,-,-,+,+,+,-,+,-,+,+,-,+,-,-,+,+,+] \\
110 & 
[0,-,-,+,-,-,-,+,-,-,-,-,+,+,+,-,-,+,-,+,+,+,+,-,-,+,-,-], \\&
[-,+,+,-,+,+,-,+,+,+,-,-,-,-,-,-,+,+,-,+,-,+,-,+,-,-,-,+] \\
114 & 
[0,-,-,+,+,+,+,-,-,+,-,+,-,+,-,-,+,+,+,-,+,+,-,-,-,-,-,-,-],\\&
[+,+,-,+,-,-,-,+,-,+,-,-,+,-,+,+,-,-,+,+,+,-,-,-,-,+,-,-,+] \\
122 & 
[0,+,-,+,-,-,+,+,+,+,+,-,-,-,+,+,-,+,-,+,+,+,+,-,-,+,-,-,-,-,-]\\&
[-,+,-,-,+,+,+,-,+,+,-,+,-,-,+,+,-,+,+,+,-,-,-,+,-,+,+,+,-,+,+]\\
\end{array} 
$$

\section{Appendix C} \label{app:C}

For integers $q=4t-1$, with $q=p^n\equiv 3 \pmod{4}$ a power of a prime $p$, we give the NG-pairs $(a,b)$ of length $v=2t\le 64$ belonging to the second Paley series. The procedure used to generate this list is described in section \ref{sec:Paley-2}. 

The sequence $a$ is quasi-symmetric and $b$ is skew-symmetric.  We record only the first $t+1$ terms of $a$ 
and the first $t$ terms of $b$. If $A$ and $B$ are the negacyclic blocks with first rows $a$ and $b$, then the 
matrix (\ref{mat-H}) is 2N-type skew-Hadamard.

$$
\begin{array}{rl}
v &  a~\&~b \text{ (truncated)} \\
\hline \\
2 & [+,-],~[+] \\
4 & [+,-,-],~[+,-] \\
6 & [+,-,-,+],~[+,+,-] \\
10 & [+,-,-,-,-,+],~[+,-,-,+,-] \\
12 & [+,-,+,+,+,-,+],~[+,-,+,+,+,+] \\
14 & [+,+,-,-,+,+,+,+],~[+,-,+,-,+,+,+] \\
16 & [+,-,-,+,+,-,+,-,-],~[+,+,+,+,-,-,+,-] \\
22 & [+,-,-,-,+,-,-,+,+,-,+,-],~[+,+,+,+,-,+,-,-,+,+,+] \\
24 & [+,-,+,+,-,+,-,-,-,+,+,+,-],~[+,-,+,-,+,+,+,+,+,+,-,-] \\
30 & [+,-,+,-,+,-,-,-,+,-,-,-,-,-,+,+],~ \\&
[+,-,-,+,-,-,-,+,+,+,+,-,+,+,-] \\
34 & [+,-,+,+,+,-,+,-,+,-,-,+,-,-,-,-,+,+],~ \\& 
[+,+,-,+,+,+,+,+,-,-,-,+,-,-,+,+,-] \\
36 & [+,-,-,-,-,-,+,+,-,+,-,+,+,+,+,-,+,+,+],~ \\& 
[+,-,+,-,-,+,+,-,-,+,-,-,-,+,+,+,-,+] \\
40 & [+,-,+,-,-,-,-,+,-,+,-,-,-,+,-,-,-,-,+,+,-],~ \\& 
[+,-,+,+,+,+,-,-,-,-,-,+,+,-,-,+,+,-,+,-] \\
42 & [+,-,+,+,+,+,+,-,-,+,-,+,+,+,+,-,+,+,-,-,+,+],~ \\& 
[+,-,-,+,-,+,-,+,-,+,+,-,-,-,-,+,-,-,-,-,-] \\
52 & [+,-,+,-,-,-,-,-,+,-,+,+,+,+,-,-,-,+,+,-,+,+,+,-,+,+,-],~ \\&
[+,-,-,+,+,+,-,-,-,-,-,+,-,-,-,-,+,-,+,-,+,-,-,+,+,-] \\
54 & [+,-,+,-,-,-,-,-,+,+,-,-,+,+,-,+,-,-,-,-,+,-,+,-,+,-,-,-],~ \\&
[+,-,+,+,+,-,-,+,+,+,-,-,-,-,-,-,-,+,+,-,+,-,-,+,+,-,+] \\
64 &
[+,-,-,-,+,-,+,-,-,+,-,-,-,+,+,-,+,+,-,-,+,+,+,+,+,+,+,-, \\&
+,+,+,-,-],~[+,-,+,+,-,-,+,-,-,-,-,+,+,+,+,+,-,+,-,+,+,-, \\&
+,-,+,+,+,-,-,-,+,-] \\
\end{array} 
$$

\section{Appendix D} \label{app:D}

For integers $q=4t-1$, with $q=p^n\equiv 3 \pmod{4}$ a power of a prime $p$, we give the NG-pairs $(a,b)$ of length 
$v=2t\le 154$ belonging to the Ito series. The procedure used to 
generate this list is described in section \ref{sec:Ito}. 
In the list below, for each length $v$, we record the primitive polynomial $f(x)$ of degree $2n$ over $\GF(p)$ used in the computation, and the NG-pair $(a,b)$.

In all cases we have $a=(+,a')$ where the subsequence $a'$ is symmetric while the whole sequence $b$ is skew-symmetric. 
We record only the first $t+1$ terms of $a$ and the first $t$ terms of $b$. If $A$ and $B$ are the negacyclic blocks with first rows $a$ and $b$, then the matrix (\ref{mat-H}) is skew-Hadamard 
of 2N-type.

Moreover, by multiplying the NG-pair $(a,b)$ by 2, we obtain in the same way a 2N-type skew-Hadamard matrix of order $1+q$.

$$
\begin{array}{rl}
v & a~\&~b \text{ (truncated)} \\
\hline \\
2 & x^2-x-1;~ p=q=3 \\&
[+,+],~ [+]\\
4 & x^2-x+3;~ p=q=7 \\&
[+,-,+],~ [+,+] \\
6 & x^2+x+7;~ p=q=11 \\&
[+,-,+,+],~ [-,-,+] \\
10 & x^2-x+2;~ p=q=19 \\&
[+,-,+,-,+,+],~ [+,+,+,-,-] \\
12 & x^2-x+7;~ p=q=23 \\&
[+,+,+,-,+,+,+],~ [+,-,+,-,-,-] \\
14 & x^6-x^5+2;~ p=3,~ q=27 \\&
[+,+,+,-,-,+,-,+],~ [-,+,-,-,-,-,-] \\
16 & x^2-x+12;~ p=q=31 \\&
[+,-,+,+,-,-,-,-,+],~ [+,+,+,-,-,+,-,+] \\
22 & x^2+x+3;~ p=q=43 \\&
[+,+,-,+,+,+,-,-,+,+,+,+],~ [-,+,-,-,+,+,+,+,-,+,-] \\
24 & x^2+x+13;~ p=q=47 \\&
[+,-,-,+,+,+,+,-,+,+,-,+,+],~ [-,+,+,-,+,-,+,-,-,-,-,-] \\
30 & x^2+x+2;~ p=q=59 \\&
[+,-,-,-,-,-,+,-,-,-,+,-,+,-,-,+],~ \\&
[-,-,+,+,+,-,+,-,-,+,-,-,-,+,+] \\
34 & x^2+x+12;~ p=q=67 \\&
[+,-,-,+,-,-,-,-,-,-,+,+,+,-,+,-,-,+] \\&
[-,-,+,+,-,-,-,+,-,-,+,-,+,-,-,-,+] \\
36 & x^2+x+11;~ p=q=71 \\&
[+,+,-,+,-,+,+,-,-,-,-,-,+,-,+,+,+,-,+], \\&
[-,-,-,+,-,-,+,-,-,-,+,+,-,-,+,+,+,+] \\
40 & x^2+x+3;~ p=q=79 \\&
[+,-,-,-,+,-,+,+,+,+,+,-,+,+,+,-,+,-,-,+,+], \\&
[-,-,-,-,+,+,-,-,+,+,-,+,-,+,+,-,+,-,-,-] \\
42 & x^2+x+2;~ p=q=83 \\&
[+,-,-,+,-,+,-,-,+,+,+,+,-,+,-,-,-,+,+,-,-,+], \\&
[-,+,-,+,-,-,-,+,-,+,+,-,-,-,-,-,-,-,-,+,+] \\
52 & x^2+x+5;~ p=q=103 \\&
[+,-,-,-,+,-,+,-,-,-,-,+,-,+,+,-,+,+,-,-,-,+,-,-,-,+,+], \\&
[-,-,+,+,-,-,-,-,-,-,-,+,-,+,+,+,-,+,-,+,+,-,+,+,-,-] \\
54 & x^2+x+5;~ p=q=107 \\&
[+,+,+,+,-,+,-,+,+,-,-,+,+,-,-,-,-,+,-,+,+,+,+,+,+,+,-,+], \\&
[-,-,-,+,+,-,-,-,-,+,+,-,+,-,+,-,+,+,-,+,+,-,-,+,-,-,-] \\
\end{array} 
$$

$$
\begin{array}{rl}
64 & x^2-x+3;~ p=q=127 \\&
[+,-,+,-,-,-,-,-,+,+,+,-,+,+,-,-,-,+,+,-,-,+,-,+,-,+,-,-, \\&
-,+,-,-,+],~[+,+,+,-,+,+,-,+,+,+,+,+,-,-,-,-,-,+,-,+,-,-, \\&
+,-,+,+,+,-,-,+,+,+] \\
66 & x^2-x+14;~ p=q=131 \\&
[+,-,-,-,-,+,+,-,+,-,+,-,-,+,+,-,-,-,+,+,-,+,-,+,+,+,+,+, \\&
-,+,+,-,-,+],~[+,+,+,+,+,+,+,-,-,-,+,-,+,+,-,+,-,-,-,+,+, \\&
-,+,-,+,+,+,-,-,-,+,-,-] \\
70 & x^2+x+2;~ p=q=139 \\&
[+,+,-,+,-,-,-,-,+,-,+,-,+,+,+,-,-,-,+,+,+,-,-,+,-,-,+,+, \\&
+,+,+,+,+,-,+,+],~[-,-,-,-,+,+,+,+,-,+,+,-,+,+,-,+,-,+,+, \\&
+,-,+,-,-,-,+,+,-,+,-,-,+,+,-,-] \\
76 & x^2+x+12;~ p=q=151 \\&
[+,-,+,+,+,-,-,+,-,-,+,-,+,+,-,+,-,-,-,+,-,+,+,+,+,+,+,-, \\&
-,+,+,-,-,-,-,-,+,-,+],~[-,+,+,+,-,-,-,+,-,+,+,-,-,-,-,-, \\&
+,-,-,-,-,+,-,+,-,+,+,-,-,-,-,+,-,-,-,+,-,-] \\
82 & x^2+x+11;~ p=q=163 \\&
[+,+,+,-,+,-,+,-,-,-,+,+,-,-,-,-,+,+,-,+,-,-,-,+,+,+,+,+, \\&
+,+,+,+,-,+,-,+,+,+,+,-,-,+],~[-,+,-,-,+,-,-,+,-,+,-,+,-, \\&
+,+,+,+,-,-,-,-,+,-,-,+,+,+,+,-,+,+,+,-,-,+,-,-,+,+,-,-] \\
106 & x^2+x+3;~ p=q=211 \\&
[+,-,-,-,+,+,-,+,-,-,-,-,-,+,-,-,-,+,+,+,+,-,-,+,+,-,-,-,-, \\&
-,-,+,-,-,+,+,-,-,-,-,+,+,-,+,+,-,+,-,+,+,-,+,-,+], \\&
[-,+,+,-,-,-,+,-,+,+,+,-,-,+,-,-,+,-,-,+,-,+,-,-,-,-,-,-,-, \\&
+,+,+,-,-,+,-,-,-,+,-,+,-,-,+,+,+,+,+,+,-,+,-,+] \\
142 & x^2+x+3;~ p=q=283 \\&
[+,+,-,+,-,-,-,-,+,-,-,-,-,-,+,+,-,-,-,-,+,+,+,-,+,+,+,-,+, \\&
-,+,-,+,-,-,+,-,+,-,-,-,-,+,+,-,+,+,-,+,+,-,+,-,-,-,+,-,-, \\&
+,+,+,+,-,+,+,-,-,+,+,-,+,+], \\&
[-,+,+,-,-,-,+,+,+,+,-,+,+,+,+,+,-,+,-,-,+,+,+,+,-,-,+,-,+, \\&
+,+,+,+,+,-,-,+,+,-,+,+,+,-,-,-,-,-,-,+,+,-,+,-,+,+,-,-,-, \\&
+,-,-,-,-,+,-,+,+,+,-,+,-] \\
154 & x^2+x+5;~ p=q=307 \\&
[+,+,+,-,-,+,-,-,+,+,+,-,+,-,+,+,-,+,-,-,+,+,-,-,+,+,-,+,-, \\&
-,-,-,-,-,+,+,-,-,-,-,-,-,-,-,-,-,+,-,-,+,-,+,+,+,+,-,+,+, \\&
-,-,-,-,+,+,-,-,-,-,+,+,+,+,-,+,-,+,-,+], \\&
[-,+,-,-,+,+,+,+,-,-,-,-,+,-,+,+,+,-,+,-,-,+,-,-,+,+,-,-,+, \\&
-,+,-,+,-,-,-,-,+,-,+,-,-,-,+,-,-,-,+,+,+,-,+,+,+,+,-,+,+, \\&
-,-,-,-,-,-,+,+,-,+,-,+,-,-,+,+,-,-,-] \\
\end{array} 
$$

\section{Appendix E} \label{app:E}

We list here the weighing matrices $W(4n,4n-2)$ of 4C-type for
odd $n\le 21$.

$$
\begin{array}{rl}
4n & a,b,c,d \\
\hline \\
4 & [0],~ [+],~ [0],~ [+] \\
12 & [0,+,+],~ [+,-,-],~ [0,-,-],~ [+,-,-] \\
20 & [0,+,+,+,+],~[+,+,-,-,+],~[0,+,-,-,+],~[+,-,+,+,-] \\
28 & 
[0,-,+,+,+,+,-],~[+,+,-,+,+,-,+], \\&
[0,-,+,-,-,+,-],~[+,+,+,-,-,+,+] \\
36 & 
[0,+,-,+,-,-,+,-,+],~[+,+,-,-,-,-,-,-,+], \\&
[0,-,+,+,-,-,+,+,-],~[+,+,+,-,+,+,-,+,+] \\
44 & 
[0,+,-,-,+,-,-,+,-,-,+],~[+,+,-,-,-,-,-,-,-,-,+], \\&
[0,+,-,+,-,+,+,-,+,-,+],~[+,+,+,-,-,+,+,-,-,+,+] \\
52 & 
[0,+,+,-,+,-,-,-,-,+,-,+,+],~[+,-,+,-,-,-,+,+,-,-,-,+,-], \\&
[0,-,-,+,+,+,-,-,+,+,+,-,-],~[+,-,-,+,-,-,-,-,-,-,+,-,-] \\
60 & 
[0,-,-,+,+,+,-,+,+,-,+,+,+,-,-],~[+,+,-,-,-,-,+,-,-,+,-,-, \\&
-,-,+],~[0,+,-,+,-,+,+,-,-,+,+,-,+,-,+], \\&
[+,-,+,+,-,-,-,-,-,-,-,-,+,+,-] \\
68 & 
[0,-,+,+,+,-,+,-,-,-,-,+,-,+,+,+,-],~[-,+,-,-,+,-,+,+,-,-, \\&
+,+,-,+,-,-,+],~[0,+,+,+,+,+,-,-,+,+,-,-,+,+,+,+,+], \\&
[+,+,+,-,+,-,-,-,+,+,-,-,-,+,-,+,+] \\
76 & 
[0,+,+,+,+,-,+,-,-,+,+,-,-,+,-,+,+,+,+],~[-,+,+,+,-,-,-,+, \\&
-,+,+,-,+,-,-,-,+,+,+],~[0,+,+,+,-,+,-,-,+,+,+,+,-,-,+,-, \\&
+,+,+],~[-,+,-,-,-,+,+,-,+,+,+,+,-,+,+,-,-,-,+] \\
84 & 
[0,-,-,+,+,-,-,+,-,+,-,-,+,-,+,-,-,+,+,-,-], \\&
[-,-,+,+,+,-,+,-,-,-,-,-,-,-,-,+,-,+,+,+,-], \\&
[0,+,-,+,-,-,+,+,-,-,-,-,-,-,+,+,-,-,+,-,+], \\&
[+,-,+,+,+,+,-,+,+,-,-,-,-,+,+,-,+,+,+,+,-] \\
\end{array} 
$$

\end{document}